\documentclass[11pt]{article}
\usepackage{amsmath,amssymb,amsthm,amsfonts,amstext,amsbsy,amscd}
\usepackage{graphics}
\usepackage{graphicx}
\usepackage{float}

\usepackage{dsfont}
\usepackage{color}

\RequirePackage[colorlinks,citecolor=blue,urlcolor=blue]{hyperref}

\RequirePackage{hypernat}

\newcommand{\PP}{\mathbb{P}}

\newcommand{\E}{\mathbb{E}}

\newcommand{\R}{\mathbb{R}}
\newcommand{\N}{\mathbb{N}}

\newtheorem{theorem}{Theorem}

\newtheorem{lemma}{Lemma}
\newtheorem{definition}{Definition}

\newtheorem{proposition}{Proposition}
\newtheorem{assumption}{Assumption}

\begin{document}
\title{Adaptive wavelet estimation of a compound Poisson process}
\author{C\'eline Duval\footnote{GIS-CREST and CNRS-UMR 8050, 3, avenue Pierre Larousse, 92245 Malakoff Cedex, France.}}
\date{}
\maketitle

\begin{abstract}
We study the nonparametric estimation of the jump density of a compound Poisson process from the discrete observation of one trajectory over $[0,T]$. We consider the microscopic regime when the sampling rate $\Delta=\Delta_T\rightarrow0$ as $T\rightarrow\infty$. We propose an adaptive wavelet threshold density estimator and study its performance for the $L_p$ loss, $p\geq 1$, over Besov spaces. The main novelty is that we achieve minimax rates of convergence for sampling rates $\Delta_T$ that vanish with $T$ at arbitrary polynomial rates. More precicely, our estimator attains minimax rates of convergence provided there exists a constant $K\geq 1$ such that the sampling rate $\Delta_T$ satisfies $T\Delta_T^{2K+2}\leq 1.$ If this condition cannot be satisfied we still provide an upper bound for our estimator. The estimating procedure is based on the inversion of the compounding operator in the same spirit as Buchmann and Gr\"ubel (2003).
\end{abstract}
\noindent \textbf{AMS 2000 subject classifications:} 62G99, 62M99, 60G50.\\
\noindent \textbf{Keywords:} Compound Poisson process, Discretely observed random process, Decompounding, Wavelet density estimation.

\section{Introduction}

\subsection{Statistical setting}

Let $R$ be a standard homogeneous Poisson process with intensity $\vartheta$ in $(0,\infty)$, we define the compound Poisson process $X$ as \begin{align*}
X_t&=\sum_{i=1}^{R_t}\xi_i,\ \ \ \ \ t\geq0
\end{align*}
where the $\big(\xi_i\big)$ are independent and identically distributed random variables and independent of the Poisson process $R$.

Assume that we have discrete observations of the process $X$ over $[0,T]$ at times $i\Delta$ for some $\Delta >0$
\begin{equation}\label{eq data}
\big(X_\Delta,\ldots,X_{\lfloor T\Delta^{-1}\rfloor\Delta}\big).\end{equation} We focus on the \emph{microscopic regime}, namely $$\Delta=\Delta_T \rightarrow 0 \ \ \ \ \ \mbox{ as }\ T\rightarrow\infty$$ and work under the following assumption. \begin{assumption}\label{ass f}
The law of the $\xi_i$ has density $f$ which is absolutely continuous with respect to the Lebesgue measure. 
\end{assumption} \noindent We denote by $\mathcal{F}(\R)$ the space of densities with respect to the Lebesgue measure supported by $\R$. We investigate the nonparametric estimation of the density $f$ on a compact interval $\mathcal{D}$ included in $\R$ from the observations \eqref{eq data}. To that end we use wavelet threshold density estimators and study their rate of convergence uniformly over Besov balls for the following loss function \begin{align}\label{eq loss}  \big(\E\big[\|\widehat {f}-f\|_{L_p(\mathcal{D})}^p\big]\big)^{1/p},\end{align} where $\widehat{f}$ is an estimator of $f$, $p\geq1$ and $$\|f\|_{L_p(\mathcal{D})}=\Big(\int_\mathcal{D} |f(x)|^pdx\Big)^{1/p}.$$ We also denote by $\|f\|_{L_p(\R)}$ the usual $L_p$ norm for $p\geq1$$$\|f\|_{L_p(\R)}=\Big(\int_\R|f(x)|^pdx\Big)^{1/p}.$$ We do not assume the intensity $\vartheta$ to be known: it is a nuisance parameter.

By Assumption \ref{ass f}, on the event $\{X_{i\Delta}-X_{(i-1)\Delta}=0\}$ no jump occurred between $(i-1)\Delta$ and $i\Delta$ and the increment $X_{i\Delta}-X_{(i-1)\Delta}$ gives no information on $f$. In the microscopic regime many increments are zero, therefore to estimate $f$ we focus on the nonzero increments and denote by $N_T$ their number over $[0,T]$. In that statistical context different difficulties arise. First the sample size $N_T$ is random. Second on the event $\{X_{i\Delta}-X_{(i-1)\Delta}\ne0\}$, the increment $X_{i\Delta}-X_{(i-1)\Delta}$ is not necessarily a realisation of the density $f$. Indeed even if $\Delta$ is small there is always a positive probability that more than one jump occurred between $(i-1)\Delta$ and $i\Delta$. Conditional on $\{X_{i\Delta}-X_{(i-1)\Delta}\ne0\}$, the law of $X_{i\Delta}-X_{(i-1)\Delta}$ has density given by (see Proposition \ref{PropDefOperator} in Section \ref{Section results} below) \begin{align}\label{eq operator 1}\mathbf{P}_\Delta[f](x)=\sum_{m=1}^\infty\PP\big(R_\Delta=m\big|R_\Delta\ne0\big)f^{\star m}(x), \ \ \ \ \mbox{ for }x\in\R,\end{align} where $\star $ is the convolution product and $f^{\star m}=f\star\ldots\star f$, $m$ times.

Adaptive estimators of the density $f$ in that statistical context already exists. Under the condition $T\Delta_T\leq 1$ (or $T\Delta_T^2\leq 1$ if $f$ is smooth enough), they attain minimax rates of convergence over Sobolev spaces for the $L_2$ loss (see Bec and Lacour \cite{Lacour}, Comte and Genon-Catalot \cite{Comte09,Comte11} and Figueroa-L\'opez \cite{Lopez}). In this paper we try to answer the following questions.
\begin{itemize}
\item [\textbf{i)}] Is it possible to construct an estimator of $f$ when $\Delta_T$ decays slowly to 0, for instance when $\Delta_T$ vanishes polynomially slowly with $T$.
\item [\textbf{ii)}] Is it possible to construct adaptive wavelet estimators that attain, over Besov spaces for the $L_p$ loss defined in \eqref{eq loss}, the classical minimax rates of convergence of the experiment where we observe $T$ independent realisations of $f$.
\end{itemize}
Without loss of generality, assuming $T$ is an integer if we observe $T$ independent realisations of a density $f$ of regularity $s$ measured with the $L_\pi$ norm, $\pi>0$, it is possible to achieve the  minimax rates of convergence for the $L_p$ loss --up to constants and logarithmic factors-- which is of the form $$T^{-\alpha(s,\pi,p)}$$  where $\alpha(s,\pi,p)\leq 1/2$ (see for instance Donoho \textit{et al.} \cite{Donoho96} and \eqref{eq alpha} hereafter). When the process $X$ is continuously observed over $[0,T]$, we have $R_T$ independent and identically distributed realisations of $f$. Moreover for $T$ large enough, $R_T$ is of the order of $T$. That is why we want compare the performance of estimators of $f$ in the regime $\Delta_T\rightarrow0$ with the classical minimax rate we would have if $X$ were continuously observed.

\subsection{Our Results\label{section our res}}

We build our estimator of $f$ using equation \eqref{eq operator 1} and proceed in two steps. The first step is the computation of the inverse of the operator $f\rightarrow \mathbf{P}_\Delta[f]$. The inverse takes the form $$\mathbf{P}_\Delta^{-1}[\nu]=\sum_{m\geq1}^\infty a_m(\vartheta,\Delta_T)\nu^{\star m},\ \ \ \ \nu\in \mathcal{F}(\R)$$ where the $\big(a_m(\vartheta,\Delta_T)\big)$ are explicit (see Proposition \ref{PropDefOperator} below). They depend on the intensity $\vartheta$ and can be estimated. We take advantage of \begin{align}\label{eq approx intro}f\approx\mathbf{L}_{\Delta,K}\big[\mathbf{P}_\Delta[f]\big],\end{align} where $\mathbf{L}_{\Delta,K}$ is the Taylor expansion of order $K$ in $\Delta$ of $\mathbf{P}_\Delta^{-1}$. It depends only on $\big({\mathbf{P}_\Delta}[f]^{\star m},m=1,\ldots,K+1\big)$. That step can be referred as decoumpounding as introduced in Buchmann \textit{et al.} \cite{Buchmann}.

The second step consists in estimating the densities $\mathbf{P}_\Delta[f]^{\star m}$, for $m=1,\ldots,K+1$. For that we use the $N_T$ nonzero increments which are independent and with density $\mathbf{P}_\Delta[f]$. The difficulty here is that $N_T$ is random. In Theorem \ref{thm Poisson} we show that conditional on $N_T$ wavelet threshold estimators of $\mathbf{P}_\Delta[f]^{\star m}$ attain a rate of convergence --up to logarithmic factors-- in $N_T^{-\alpha(s,\pi,p)}$. For $T$ large enough we prove (see Proposition \ref{prop control rate} in Section \ref{Section proof}) that $N_T$ concentrates around a deterministic value of the order of $T$, giving an unconditional rate of convergence in $T^{-\alpha(s,\pi,p)}$. We inject those estimators into $\mathbf{L}_{\Delta,K}$, defined in \eqref{eq approx intro}, and obtain an estimator of $f$ that we call \textit{estimator corrected at order $K$.}

The study of the rate of convergence of the estimator corrected at order $K$ requires to control two distinct error terms. A deterministic one due the first step which is the error made when approximating $f$ by $\mathbf{L}_{\Delta,K}\big[\mathbf{P}_\Delta[f]\big]$ in \eqref{eq approx intro}. And a statistical one due to the replacement of the ${\mathbf{P}_\Delta}[f]^{\star m}$ by estimators in the second step. The deterministic error decreases when $K$ increases, then the idea is to choose $K$ sufficiently large for the deterministic error term to be negligible in front of the statistical one. We give in Theorem \ref{thm Poisson} an upper bound for the rate of convergence of the estimator corrected at order $K$ which is in --up to constants and logarithmic factors--
$$\max\{T^{-\alpha(s,\pi,p)},\Delta_T^{K+1}\}.$$ It decreases with $K$ and if there exists $K_0$ such that \begin{align}\label{eq cond K}T\Delta_T^{2K_0+2}\leq 1,\end{align} since $\alpha(s,\pi,p)\leq 1/2$ the estimator corrected at order $K_0$ attains the minimax rates of convergence. It follows that for every $\Delta_T$ polynomially decreasing with $T$, it is possible to exhibit $K_0$ such that \eqref{eq cond K} is valid and the estimator corrected at order $K_0$ provides a positive answer to \textbf{i)} and \textbf{ii)}. If no $K$ enables to verify condition \eqref{eq cond K}, Theorem \ref{thm Poisson} provides an upper bound for the rate of convergence of the estimator corrected at order $K$, in that case the estimator still provide a positive answer to \textbf{i)}.

In the case of a compound Poisson processes, the results of the present paper generalise to some extend those of Bec and Lacour \cite{Lacour}, Comte and Genon-Catalot \cite{Comte09,Comte11} and Figueroa-L\'opez \cite{Lopez}. This is discussed in further details in Section \ref{section discuss}. In Section \ref{Section results} we give the main results of the paper. We properly define wavelet functions and Besov spaces used for the estimation before having a complete construction of the estimator corrected at order $K$. Then we give an upper bound for its rate of convergence for the $L_p$ loss defined in \eqref{eq loss}, $p\geq1$, uniformly over Besov balls. A numerical example illustrates the behavior of the estimator corrected at order $K$ in Section \ref{Section Num Ex}. Finally Section \ref{Section proof} is dedicated to the proofs.

\

The model of this paper is central in many application fields \textit{e.g.} statistical physics (see Moharir \cite{MOHARIR}), biology (see Huelsenbeck \textit{et al.} \cite{Huelsenbeck}), financial series or mathematical insurance (see Scalas \cite{Scalas}). It is well adapted to study phenomena where random independent events occur at random times. For instance, in insurance failure theory these events can model the claims that insurance companies have to pay to the subscribers. The insurer's surplus at a given time $t$ can be modeled by the following process
\begin{align*}
K(t)=K_0+kt-X_t,
\end{align*} where $K_0$ is the capital of the company at time 0, the second term is a deterministic trend corresponding to the average income received from the subscribers and $X$ is a compound Poisson process modeling the insurance claims occurring at random times with random amount of money at stake. It is the Cram\'er-Lundberg model; see Embrechts \textit{et al.} \cite{Embrechts} or Scalas \cite{Scalas}. Compound Poisson processes can also model the changes of an asset price in finance; see Masoliver \textit{et al.} \cite{CTRW}.

\section{Main results}\label{Section results}

\subsection{Besov spaces and wavelet thresholding}
To estimate the densities $\big({\mathbf{P}_\Delta}[f]^{\star m},m=1,\ldots,K+1\big)$ we use wavelet threshold density estimators and study their performance uniformly over Besov balls. In this paragraph we reproduce some classical results on Besov spaces, wavelet bases and wavelet threshold estimators (see Cohen \cite{Cohen}, Donoho \textit{et al.} \cite{Donoho96} or Kerkyacharian and Picard \cite{KP00}) that we use in the next sections.

\subsubsection*{Wavelets and Besov spaces}
We describe the smoothness of a function with Besov spaces on $\mathcal{D}$. We recall here some well documented results on Besov spaces and their connection to wavelet bases (see Cohen \cite{Cohen}, Donoho \textit{et al.} \cite{Donoho96} or Kerkyacharian and Picard \cite{KP00}). Let $\big(\psi_{\lambda}\big)_\lambda$ be a regular wavelet basis adapted to the domain $\mathcal{D}$. The multi-index $\lambda$ concatenates the spatial index and the resolution level $j=|\lambda|$. Set $\Lambda_j:=\{\lambda,|\lambda|=j\}$ and $\Lambda=\cup_{j\geq -1}\Lambda_j$, for $f$ in $L_p(\R)$ we have
\begin{align}\label{eq fdecomp1}
f&=\sum_{j\geq-1}\sum_{\lambda\in\Lambda_j}\langle f,\psi_\lambda\rangle\psi_\lambda,
\end{align} where $j=-1$ incorporates the low frequency part of the decomposition and $\langle .,\rangle$ denotes the usual $L_2$ inner product. We define Besov spaces in term of wavelet coefficients, for $s>0$ and $\pi \in (0,\infty]$ a function $f$ belongs to the Besov space $\mathcal{B}^s_{{\pi} \infty}(\mathcal{D})$ if the norm
\begin{align}\label{eq besov}
\|f\|_{\mathcal{B}^s_{{\pi} \infty}(\mathcal{D})}&:=\underset{j\geq-1}{\sup}2^{j(s+1/2-1/\pi)}\Big(\sum_{\lambda\in\Lambda_j}|\langle f,\psi_\lambda\rangle|^\pi\Big)^{1/\pi}
\end{align} is finite, with usual modifications if $\pi=\infty$.

We need additional properties on the wavelet basis $\big(\psi_{\lambda}\big)_\lambda$, which are listed in the following assumption.
\begin{assumption}\label{Ass} For $p\geq1$,
\begin{itemize}
\item We have for some $\mathfrak{C}\geq 1$ $$\mathfrak{C}^{-1}2^{|\lambda|(p/2-1)}\leq \|\psi_\lambda\|_{L_p(\mathcal{D})}^p\leq \mathfrak{C} 2^{|\lambda|(p/2-1)}.$$
\item For some $\mathfrak{C}>0$, $\sigma >0$ and for all $s\leq\sigma$, $J\geq0$, we have \begin{align}\label{eq ass1} \big\|f-\sum_{j\leq J}\sum_{\lambda\in\Lambda_j} \langle f,\psi_\lambda\rangle\psi_\lambda\big\|_{L_p(\mathcal{D})}\leq \mathfrak{C}2^{-Js}\|f\|_{\mathcal{B}^s_{{\pi} \infty}(\mathcal{D})}.\end{align}
\item If $p\geq 1$, for some $\mathfrak{C}\geq 1$ and for any sequence of coefficients $\big(u_\lambda\big)_{\lambda\in\Lambda}$, \begin{align}\label{eq ass2}\mathfrak{C}^{-1}\Big\|\sum_{\lambda\in\Lambda}u_\lambda\psi_\lambda\Big\|_{L_p(\mathcal{D})}\leq \Big\|\Big( \sum_{\lambda\in\Lambda}|u_\lambda\psi_\lambda|^2\Big)^{1/2}\Big\|_{L_p(\mathcal{D})} \leq\mathfrak{C}\Big\|\sum_{\lambda\in\Lambda}u_\lambda\psi_\lambda\Big\|_{L_p(\mathcal{D})}.\end{align}
\item For any subset $\Lambda_0\subset\Lambda$ and for some $\mathfrak{C}\geq 1$ \begin{align}\label{eq ass3} \mathfrak{C}^{-1} \sum_{\lambda\in\Lambda_0}\|\psi_\lambda\|_{L_p(\mathcal{D})}^p\leq \int_{\mathcal{D}}\Big(\sum_{\lambda\in\Lambda_0}|\psi_\lambda(x)|^2\Big)^{p/2}\leq \mathfrak{C}\sum_{\lambda\in\Lambda_0}\|\psi_\lambda\|_{L_p(\mathcal{D})}^p .\end{align}
\end{itemize}
\end{assumption}

Property \eqref{eq ass1} ensures that definition \eqref{eq besov} of Besov spaces matches the definition in terms of linear approximation. Property \eqref{eq ass2} ensures that $\big(\psi_{\lambda}\big)_\lambda$ is an unconditional basis of $L_p$ and \eqref{eq ass3} is a super-concentration inequality (see Kerkyacharian and Picard \cite{KP00} p. 304 and p. 306).

\subsubsection*{Wavelet threshold estimator} Let $(\phi,\psi)$ be a pair of scaling function and mother wavelet that generate a basis $\big(\psi_{\lambda}\big)_\lambda$ satisfying Assumption \ref{Ass} for some $\sigma>0$. We rewrite \eqref{eq fdecomp1}
\begin{align*}
f&=\sum_{k\in \Lambda_0}\alpha_{0k}\phi_{0k}+\sum_{j\geq 1}\sum_{k\in\Lambda_j}\beta_{jk}\psi_{jk},
\end{align*} where $\phi_{0k}(\bullet)=\phi(\bullet-k)$ and $\psi_{jk}(\bullet)=2^{j/2}\psi(2^j\bullet-k)$ and \begin{align*} \alpha_{0k}&=\int \phi_{0k}(x)f(x)dx\\ \beta_{jk}&=\int \psi_{jk}(x)f(x)dx.\end{align*} For every $j\geq 0$, the set $\Lambda_j$ has cardinality $2^j$ and incorporates boundary terms that we choose not to distinguish in the notation for simplicity. An estimator of a function $f$ is obtained when replacing the $(\alpha_{0k})$ and $(\beta_{jk})$ by estimated values. In the sequel we uses $(\gamma_{jk})$ to design either $(\alpha_{0k})$ or $(\beta_{jk})$ and $(g_{jk})$ for the wavelet functions $(\phi_{0k})$ or $(\psi_{jk})$.

We consider classical hard threshold estimators of the form
\begin{align*}
\widehat{f}(\bullet)&= \sum_{k\in \Lambda_0}\widehat{\alpha_{0k}}\phi_{0k}(\bullet)+\sum_{j= 1}^J\sum_{k\in\Lambda_j}\widehat{\beta_{jk}}\mathds{1}_{\big\{|\widehat{\beta_{jk}}|\geq\eta\big\}}\psi_{jk}(\bullet),\end{align*} where $\widehat{\alpha_{0k}}$ and $\widehat{\beta_{jk}}$ are estimators of $\alpha_{0k}$ and $\beta_{jk}$, $J$ and $\eta$ are respectively the resolution level and the threshold, possibly depending on the data. Thus to construct $\widehat{f}$ we have to specify estimators $(\widehat{\gamma_{jk}})$ of the $(\gamma_{jk})$ and the coefficients $J$ and $\eta$.

\subsection{Construction of the estimator}

 Assume that we have $\lfloor T\Delta^{-1}\rfloor$ discrete data at times $i\Delta$ for some $\Delta >0$ of the process $X$
\begin{equation*}
\big(X_\Delta,\ldots,X_{\lfloor T\Delta^{-1}\rfloor\Delta}\big).
\end{equation*} Introduce the increments
\begin{align*}
\mathbf{D}^\Delta X_i=X_{i\Delta}-X_{(i-1)\Delta},\ \ \ \mbox{ for } i=1,\dots,\lfloor T\Delta^{-1}\rfloor,\end{align*}
 where $X_0=0$. They are independent and identically distributed since $X$ is a compound Poisson process. Define
\begin{align*}S_1&=\inf\big\{j, \mathbf{D}^\Delta X_j\ne0\big\}\wedge T\\S_i&=\inf\big\{j>S_{i-1}, \mathbf{D}^\Delta X_j\ne0\big\}\wedge T\ \ \ \mbox{for }i\geq1,\end{align*} where $S_i$ is the random index of the $i$th jump and
$$N_T=\sum_{i=1}^{\lfloor T\Delta^{-1}\rfloor}\mathds{1}_{\{\mathbf{D}^\Delta X_i\ne0\}}
$$ the random number of nonzero increments observed over $[0,T]$. By Assumption \ref{ass f}, on the event $\{\mathbf{D}^\Delta X_i=0\},$ no jump occurred between $(i-1)\Delta$ and $i\Delta$. In the microscopic regime when $\Delta=\Delta_T\rightarrow 0$ as $T$ goes to infinity many increments are null and convey no information on $f$, hence for the estimation of $f$ we focus on the nonzero ones $$\big(\mathbf{D}^\Delta X_{S_1},\ldots,\mathbf{D}^\Delta X_{S_{N_T}}\big).$$

\begin{proposition}\label{PropDefOperator}
The distribution of the increment $\mathbf{D}^\Delta X_{S_1}$ has density with respect to the Lebesgue measure given by
\begin{align*}\mathbf{P}_\Delta[f]&=\sum_{m=1}^\infty p_m(\Delta) f^{\star m},
\end{align*}
 where $$p_m(\Delta)=\PP\big(R_\Delta=m|R_\Delta\ne0\big)=\frac{1}{e^{\vartheta\Delta}-1}\frac{\big(\vartheta\Delta\big)^m}{m!}.$$
 Let $\Delta_0$ be such that $$\sum_{m=2}^\infty \frac{\big(\vartheta\Delta_0\big)^{m-2}}{m!}\leq1.$$ For $\Delta\leq\Delta_0$, we have that $$1- \vartheta\Delta\leq p_1(\Delta)\leq1.$$
\end{proposition}
 It is straightforward to verify that the nonlinear operator $\mathbf{P}_\Delta$ is a mapping from $\mathcal{F}(\R)$ to itself. The observations $\big(\mathbf{D}^\Delta X_{S_i}\big)$ are realisations of the density $\mathbf{P}_\Delta[f]$ and by Proposition \ref{PropDefOperator} the weight $p_1(\Delta)\rightarrow1$ in the limit $\Delta=\Delta_T\rightarrow0$. It follows that for $\Delta_T$ small enough most of the $\big(\mathbf{D}^\Delta X_{S_i}\big)$ have distribution $f$. Then a naive method to estimate $f$ is to apply classical density estimators to the $\big(\mathbf{D}^\Delta X_{S_i}\big)$. That estimator requires a convergence condition on $\Delta_T$ to achieve minimax rate of convergence (see Theorem \ref{thm Poisson}). However we wish to construct an estimator that attains minimax rates of convergence with weaker conditions on $\Delta_T$.

We adopt the estimating strategy of section \ref{section our res} and construct an approximation of $f$.
\begin{lemma}\label{lem Inverse T}
The inverse $\mathbf{P}_\Delta^{-1}$ of $\mathbf{P}_\Delta$, such that for all densities $f$ in $\mathcal{F}(\R)$ if $ \mathbf{P}_\Delta[f]=\nu$ we have $\mathbf{P}_\Delta^{-1}[\nu]=f$, is given by
\begin{align*}
\mathbf{P}_\Delta^{-1}[\nu]= \frac{1}{\vartheta\Delta}\sum_{m=1}^\infty \frac{(-1)^{m+1}}{m}(e^{\vartheta\Delta}-1)^m\nu^{\star m}.
\end{align*}
\end{lemma}
\noindent To build the estimator corrected at order $K$ we use that $\mathbf{P}_\Delta^{-1}$ is a power series whose coefficients are equivalent to increasing powers of $\Delta$. Then $\mathbf{L}_{\Delta,K}$ the Taylor expansion of order $K$ in $\Delta$ of $\mathbf{P}_\Delta^{-1}$ is obtained by keeping the first $K+1$ terms of the inverse \begin{align}\label{eq LTaylor}\mathbf{L}_{\Delta,K}[\nu]=\frac{1}{\vartheta\Delta}\sum_{m=1}^{K+1} \frac{(-1)^{m+1}}{m}(e^{\vartheta\Delta}-1)^m\nu^{\star m},\ \ \ \nu \in \mathcal{F}(\R).\end{align} Next we construct wavelet threshold density estimators of the first $K+1$ convolution powers of ${\mathbf{P}_\Delta}[f]$ that will be plugged in \eqref{eq LTaylor}. Define
\begin{align}\label{eq est Coeffconvol}
\widehat{\gamma}^{(m)}_{jk}&=\frac{1}{N_{T,m}}\sum_{i=1}^{N_{T,m}}g_{jk} \Big(\mathbf{D}^{\Delta}_m X_{S_i}\Big)\ \ \ \ m\geq1,\end{align} where $N_{T,m}=\big\lfloor N_T/m \big\rfloor\geq 1$ for large enough $T$ and $$\mathbf{D}^{\Delta}_m X_{S_i}=\mathbf{D}^\Delta X_{S_i}+ \mathbf{D}^\Delta X_{S_{N_{T,m}+i}}+\dots+\mathbf{D}^\Delta X_{S_{(m-1)N_{T,m}+i}}.$$ The $\big(\mathbf{D}^\Delta X_{S_i}\big)$ are independent and identically distributed with density $\mathbf{P}_\Delta[f]$, thus the $\big(\mathbf{D}^{\Delta}_m X_{S_i}\big)$ are independent and identically distributed with density $\mathbf{P}_\Delta[f]^{\star m}$. Let $\eta>0$ and $J\in \N \setminus\{0\},$ define $\widehat{P_{\Delta,m}}$ the estimator of $\mathbf{P}_\Delta[f]^{\star m}$ over $\mathcal{D}$ \begin{align}\label{eq est convol}
\widehat{P_{\Delta,m}}(x)&=\sum_{k}\widehat{\alpha}_{0k}^{(m)}\phi_{0k}(x)+\sum_{j= 0}^J\sum_{k}\widehat{\beta}_{jk}^{(m)}\mathds{1}_{\big\{|\widehat{\beta}_{jk}^{(m)}|\geq \eta\big\}}\psi_{jk}(x),\ \ \ x\in\mathcal{D}.
\end{align}

\begin{definition}\label{def est f Poiss bis}
We define $\widetilde{f}^K_{T,\Delta}$ the estimator corrected at order $K$ for $K$ in $\N$ and $x$ in $\mathcal{D}$ as
 \begin{align}
\widetilde{f}^K_{T,\Delta}(x)&=\sum_{m=1}^{K+1}\frac{(-1)^{m+1}}{m}\frac{\big(e^{\widehat{\vartheta}_T\Delta}-1\big)^m}{ \widehat{\vartheta}_T\Delta}\widehat{P_{\Delta,m}}(x)\label{eq Est f Threshold Corr bis},
 \end{align} where \begin{align}\label{eq est tau0}\widehat{\vartheta}_T&=-\frac{1}{\Delta}\log(1-\widehat{p}_T)\end{align} and $$\widehat{p}_T=\frac{N_T}{\lfloor T\Delta^{-1}\rfloor}$$ is the empirical estimator of $p(\Delta)=\PP(R_\Delta=0)=1-e^{-\vartheta{\Delta}}.$
\end{definition} \noindent Lemma \ref{lem Inverse T} justifies the form of the estimator corrected at order $K$.

\subsection{Convergence rates}
We estimate densities $f$ which verify a smoothness property in term of Besov balls $$\mathcal{F}(s,{\pi},\mathfrak{M})=\big\{ f\in \mathcal{F}(\R), \|f\|_{\mathcal{B}^s_{{\pi} \infty}(\mathcal{D})}\leq\mathfrak{M}\big\},$$ where $\mathfrak{M}$ is a positive constant. We are interested in estimating $f$ on the compact interval $\mathcal{D}$, that is why we only impose that its restriction to $\mathcal{D}$ belongs to a Besov ball.

\begin{theorem}  \label{thm Poisson} We work under Assumptions \ref{ass f} and \ref{Ass}. Let $\sigma>s>1/\pi$, $p\geq1\wedge \pi$ and $\widehat{P_{\Delta_T,m}}$ be the threshold wavelet estimator of $\mathbf{P}_{\Delta_T}[f]^{\star m}$ on $\mathcal{D}$ constructed from $(\phi,\psi)$ and defined in \eqref{eq est convol}. Take $J$ such that $$2^JN_T^{-1}\log\big(N_T^{1/2}\big)\leq 1,$$  and $$\eta=\kappa N_T^{-1/2}\sqrt{\log\big(N_T^{1/2}\big)},$$ for some $\kappa>0$. Let
\begin{align}\label{eq alpha}\alpha(s,p,\pi)=\min\Big\{\frac{s}{2s+1},\frac{s+1/p-1/{\pi}}{2\big(s+1/2-1/{\pi}\big)}\Big\}.\end{align} \begin{enumerate} \item[1)] The estimator $\widehat{P_{\Delta_T,m}}$ verifies for large enough $T$ and sufficiently large $\kappa>0$\begin{align*}
\underset{\mathbf{P}_{\Delta_T}[f]^{\star m}\in\mathcal{F}(s,{\pi},\mathfrak{M})}{\sup}\big(\E \big[\big\|\widehat{P_{\Delta_T,m}} -\mathbf{P}_{\Delta_T}[f]^{\star m}\big\|_{L_p(\mathcal{D})}^p\big|N_T\big] \big)^{1/p}&\leq \mathfrak{C}  N_T^{-{\alpha(s,p,\pi)}},
\end{align*} up to logarithmic factors in $T$ and where $\mathfrak{C}$ depends on $s,\pi,p,\mathfrak{M},\phi,\psi$.

\item[2)] The estimator corrected at order $K$ $\widetilde{f}^K_{T,\Delta_T}$ defined in \eqref{eq Est f Threshold Corr bis} verifies for $T$ large enough and any positive constants $\underline{\mathfrak{T}}$ and $\overline{\mathfrak{T}}$
\begin{align*}
\underset{\vartheta\in [\underline{\mathfrak{T}},\overline{ \mathfrak{T}}]}{\sup}\underset{f\in\mathcal{F}(s,{\pi},\mathfrak{M})}{\sup}\big(\E \big[\big\| \widetilde{f}^K_{T,\Delta_T} -f\big\|_{L_p(\mathcal{D})}^p\big] \big)^{1/p}&\leq \mathfrak{C} \max\big\{T^{-{\alpha(s,p,\pi)}},\Delta_T^{K+1}\big\},
\end{align*} up to logarithmic factors in $T$ and where $\mathfrak{C}$ depends on $s,\pi,p,\mathfrak{M},\phi,\psi,\underline{\mathfrak{T}},$ $\overline{ \mathfrak{T}},K$.\end{enumerate}
\end{theorem} \noindent The proof of Theorem \ref{thm Poisson} is postponed to Section \ref{Section proof}. From a practical point of view when one computes the estimator $\widetilde{f}^K_{T,\Delta_T}$ from \eqref{eq data} the sample size is $N_T$, which is why in Theorem \ref{thm Poisson} we give the resolution level $J$ and the threshold $\eta$ as functions of $N_T$ instead of replacing $N_T$ by its deterministic counterpart. Explicit bound for $\kappa$ is given in Lemma \ref{lem BernConvol Poisson} hereafter.

In practice the values $T$ and $\Delta_T$ are imposed or chosen by the practitioner. Theorem \ref{thm Poisson} ensures that the estimator corrected at order $K$ attains the minimax rate $T^{-{\alpha(s,p,\pi)}}$ for the smallest $K$ such that $$\Delta_T=O\big(T^{-\frac{{\alpha(s,p,\pi)}}{K+1}}\big).$$ Since $\alpha(s,p,\pi)\leq 1/2$ it is sufficient to choose $K$ such that $$T\Delta_T^{2K+2}=O(1).$$ If $\Delta_T$ decays as a power of $T$ \textit{i.e.} if there exists $\delta>0$ such that for some $\mathfrak{C}>0$ $$\Delta_T\leq \mathfrak{C}T^{-\delta},$$
 it is always possible to find a correction level $K$ satisfying the previous constraint. The case $K=0$ corresponds to the uncorrected estimator; it is the naive estimator one would compute making the approximation $f\approx\mathbf{P}_\Delta[f]$. In that case we get a rate of convergence in $$\max\{T^{-{\alpha(s,p,\pi)}},\Delta_T\},$$ which attains the minimax rate if $T^{\alpha(s,p,\pi)}\Delta_T\leq1$. Since $\alpha(s,\pi,p)\leq 1/2$, it follows that the condition $T^{\alpha(s,p,\pi)}\Delta_T\leq1$ already improves on the condition $T\Delta_T^2\leq1$ of Bec and Lacour \cite{Lacour}, Comte and Genon-Catalot \cite{Comte09,Comte11} or Figueroa-L\'opez \cite{Lopez} (see Section \ref{section discuss} for comparison with other works).

\section{A numerical example\label{Section Num Ex}}

We illustrate the behaviour of the estimator corrected at order $K$ when $K$ increases and compare its performance with an oracle: the wavelet estimator we would compute in the idealised framework where all the jumps are observed
\begin{align*}
\widehat{f}^{Oracle}(x)=\sum_{k}\widehat{\alpha}_{0k}^{Oracle}\phi_{0k}(x)+\sum_{j= 0}^J\sum_{k}\widehat{\beta}_{jk}^{Oracle}\mathds{1}_{\big\{|\widehat{\beta}_{jk}^{Oracle}|\geq \eta\big\}}\psi_{jk}(x),\end{align*} where $$\widehat{\alpha}_{0k}^{Oracle}=\frac{1}{R_T}\sum_{i=1}^{R_T}\phi_{0k}(\xi_i)\ \ \ \mbox{and}\ \ \ \widehat{\beta}_{jk}^{Oracle}=\frac{1}{R_T}\sum_{i=1}^{R_T}\phi_{0k}(\xi_i),$$ $R_T$ being the value of the Poisson process $R$ at time $T$ and $(\xi_i)$ the jumps. The parameters $J$ and $\eta$ as well as the wavelet bases $(\phi,\psi)$ are the same as those used to compute the estimator corrected at order $K$.
We consider a compound Poisson process of intensity $\vartheta=1$ on $[0,T]$ and of compound law
\begin{align*}
f(x)=(1-a) f_1(x)+af_2(x)
\end{align*} where $f_1$ is the density of a Gaussian $\mathcal{N}(0,1)$ and $f_2$ of a Laplace with location parameter 1 and scale parameter 0.1, we take $a=0.05$.

\begin{figure}[H]
\begin{center}
 \includegraphics[scale=0.4]{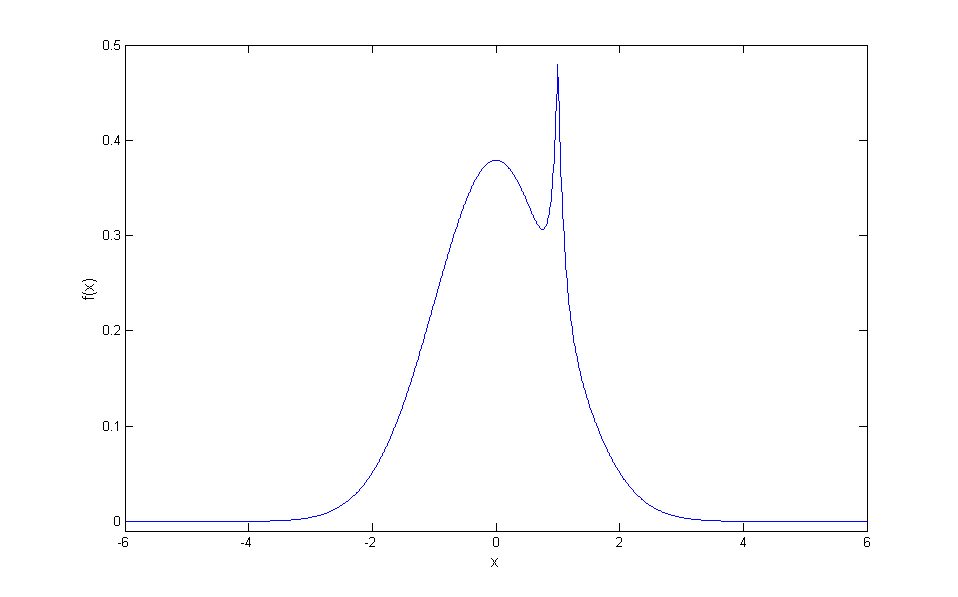}
 \caption{\footnotesize{Density $f:$ $f(x)=0.95f_1(x)+0.05f_2(x)$ $x\in [-6,6]$.}}\label{Fig loi}
 \end{center}
\end{figure}
 \noindent We estimate the mixture $f$ (see Figure \ref{Fig loi}) on $\mathcal{D}=[-6,6]$ with the estimator corrected at order $K$ for different values of $K$ and study the results with the $L_2$ error. We also compare them with the oracle $\widehat{f}^{Oracle}$. Wavelet estimators are based on the evaluation of the first wavelet coefficients, to perform those we use Symlets 4 wavelet functions and a resolution level $J=10$. Moreover we transform the data in an equispaced signal on a grid of length $2^L$ with $L=8$, it is the binning procedure (see H\"{a}rdle \textit{et al.} \cite{KerkPicTsyb} Chap. 12). The threshold is chosen as in Theorem \ref{thm Poisson}. The estimators we obtain take the form of a vector giving the estimated values of the density $f$ on the uniform grid $[-6,6]$ with mesh $0.01$. We use the wavelet toolbox of \textsf{Matlab}.

Figure \ref{Fig Comp1} represents the corrected estimator for $K=0$ and $K=1$ and the oracle. All the estimators are evaluated on the same trajectory. They manage to reproduce the shape of the density $f$. As expected the oracle looks better than the other two and the uncorrected ($K=0$) seems to make larger errors than the 1-corrected in estimating $f$. Figure \ref{Fig Comp2} represents for every values in $[-6,6]$ the absolute distance between those estimators --evaluated on the same trajectory-- and the true density $f$. Therefore it enables to determine in which area an estimator fails to estimate $f$ and to get an idea of the error made. The graphic was obtained after $M=1000$ Monte-Carlo simulations of each estimator and averaging the results. The uncorrected estimator is not as good as the estimator corrected at order 1. The oracle and the estimator corrected at order 1 seem to have similar performances. Each of the estimators makes larger errors around $1$ which is where the density $f$ is peaked.

\begin{figure}[H]
\begin{center}
\includegraphics[scale=0.4]{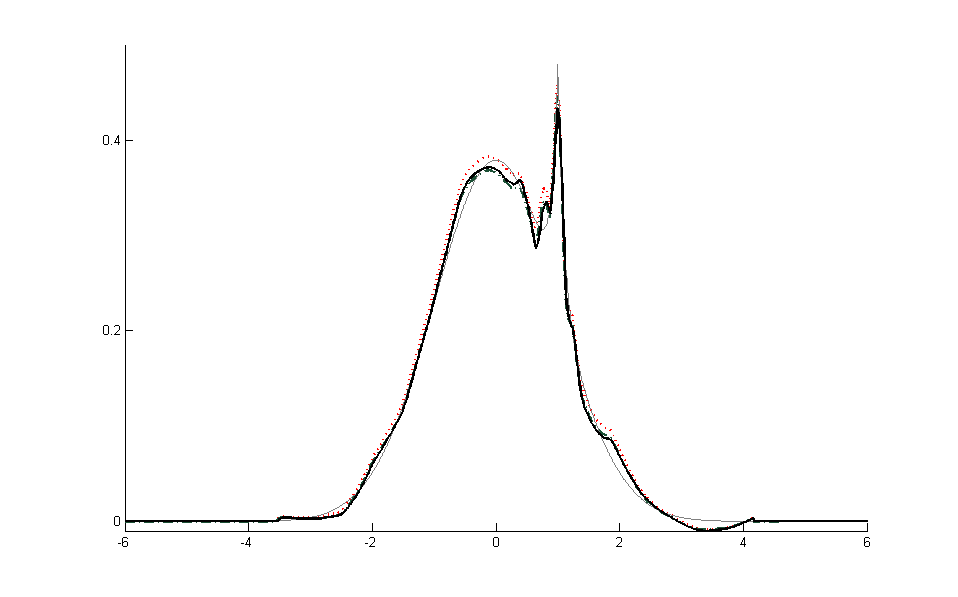}\hfill
\caption{Estimators of the density $f$ (plain grey) for $T=10000$ and $\Delta=0.1$: the uncorrected (dotted red), the 1-corrected (dashed green) and the oracle (plain dark). }\label{Fig Comp1}
\end{center}
\end{figure}

\begin{figure}[H]
\begin{center}
\includegraphics[scale=0.4]{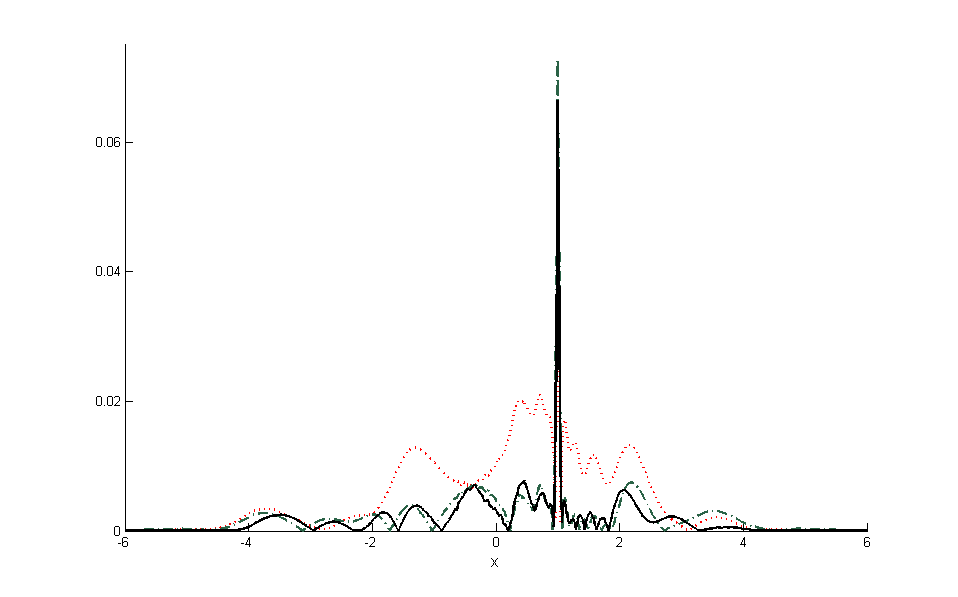}
\caption{Mean absolute error between the estimators and the true density (M=1000, $T=10000$ and $\Delta=0.1$): the uncorrected (dotted red), the 1-corrected (dashed green) and the oracle (plain dark). }\label{Fig Comp2}
\end{center}
\end{figure}
Evaluation of the $L_2$ errors enables to confirm the former graphical observation. We approximate the $L_2$ errors by Monte Carlo. For that we compute $M=1000$ times each estimator (for $T=10000$ and $\Delta=0.1$) and approximate the $L_2$ loss by$$\frac{1}{M}\sum_{i=1}^M\Big(\sum_{p=0}^{1200} \big(\widehat{f}(-6+0.01 p)-f(-6+0.01 p)\big)^2\times0.01\Big),$$ where $\widehat{f}$ is one of the estimators. For each Monte Carlo iteration the corrected and oracle estimators are evaluated on the same trajectory. The results are reproduced in the following table. \begin{center}
\begin{tabular}{|l|c|c|c|c|c|}
  \hline
  Estimator & Oracle & $K=0$ & $K=1$ & $K=2$ & $K=3$\\
  \hline
  $L_2$ error \small{($\times 10^{-4}$)} & 0.1117 & 0.1842 &   0.1353  &  0.1350   & 0.1350     \\
  \hline
  Standard deviation \small{($\times 10^{-5}$)} &  0.3495  & 0.4434  &   0.4363  &  0.4366   & 0.4366     \\
  \hline
\end{tabular}
\end{center} This confirms that there is an actual gain in considering the estimator corrected at order 1 instead of the uncorrected one. In the following table we estimate the $\big(p_m(\Delta)\big)$ defined in Proposition \ref{PropDefOperator}. \begin{center}
\begin{tabular}{|l|c|c|c|}
  \hline
  Estimated quantity & $\widehat{p_1}$ & $\widehat{p_2}$ & $\widehat{p_3}$\\
  \hline
  Estimation & 0.9508  &  0.0476  &  0.0016   \\
  \hline
  Standard deviation  &  0.0022  & 0.0022  &   0.0004   \\
  \hline
\end{tabular}
\end{center} It turns out that without the correction we estimate the density $f$ on a data set where $5\%$ of the observations are realisations of a law which is not $f$. This explains why it is relevant to take them into account when estimating $f$. Considering more than 1 or 2 corrections is unnecessary as the $L_2$ losses get stable afterwards. The $L_2$ loss of the oracle is strictly lower than the loss of the estimator corrected at order $K$, even for large $K$. That difference is explained by the fact that to estimate the $m$th convolution power we do not use $N_T$ data points but $N_{T,m}=\lfloor N_T/m\rfloor$. Therefore we do not loose in terms of rate of convergence, but we surely deteriorate the constants in comparison with the oracle.
Numerical results are consistent with the theoretical results of Theorem \ref{thm Poisson} where we proved a rate of convergence for the estimator corrected at order $K$ in $$\max\big\{T^{-{\alpha(s,p,\pi)}}, \Delta_T^{K+1}\big\}.$$ Since ${\alpha(s,p,\pi)}\leq1/2$, the rate decreases with $K$ and becomes stable once $\Delta_T^{2K+2}T\leq \mathfrak{C}$. In the numerical example we took $T=10000$ and $\Delta=0.1$ thus $T\Delta^4=1$ which explains why in the example we did not observe improvements when correcting with $K$ greater than 2.

\section{Discussion}\label{section discuss}
\subsection{Relation to other works}

A compound Poisson process is a pure jump L\'evy process and can be studied accordingly using L\'evy-Kintchine formula. Estimating the jump density $f$ is then equivalent to estimating the L\'evy measure since for compound Poisson process it is the product $\vartheta f(x)dx$. A possible estimation strategy in that case is to provide an estimator of the Fourier transform of the density. That strategy is quite different from the one introduced in this paper but is usually adopted when estimating the compound law of a compound Poisson process (see Figueroa-L\'opez \cite{Lopez}, Comte and Genon-Catalot \cite{Comte09,Comte11} or Bec and Lacour \cite{Lacour}).

The nonparametric estimation of the L\'evy measure from the discrete observation of a pure jump L\'evy process from high frequency data (which corresponds to our microscopic regime $\Delta_T\rightarrow0$) has been studied in great detail by Comte and Genon-Catalot \cite{Comte09,Comte11} and Figueroa-L\'opez \cite{Lopez}. In \cite{Lopez} the nonparametric estimation of the L\'evy density is made via a sieve estimator. They show that it attains minimax rates of convergence for the $L_2$ loss uniformly over a class of Besov functions for a sampling size $\Delta_T$ such that --with our notation-- $T\Delta_T\leq1$. Comte and Genon-Catalot \cite{Comte09,Comte11} construct an adaptive nonparametric estimator of the L\'evy measure, which attains minimax rates of convergence on Sobolev spaces for the $L_2$ loss for a sampling size $\Delta_T$ such that $T\Delta_T\leq 1$ (or $T\Delta_T^2\leq1$ under smoother assumptions). Bec and Lacour \cite{Lacour} obtained similar results when $T\Delta_T^2\leq1$. The statistical setting of \cite{Comte11} is more general since they estimate the L\'evy measure from observations of a L\'evy process with a Brownian component.

Our result is limited to the Poisson case contrary to Bec and Lacour \cite{Lacour}, Comte and Genon-Catalot \cite{Comte09} and Figueroa-L\'opez \cite{Lopez} who worked on the larger class of pure jump L\'evy processes. However in the case of a Poisson process we generalise them since we provide an adaptive density estimator which attains minimax rates of convergence, for the $L_p$ loss, $p\geq 1$, uniformly over Besov balls for regime where $\Delta_T$ is polynomially slow. If $\Delta_T$ decays even slower, for instance logarithmically in $T$, we still have an upper bound for the rate of convergence of our estimator.

\subsection{Possible extensions}

In this paper we give an adaptive minimax procedure for the estimation of the compound density of a compound Poisson process in the microscopic regime. The same estimation problem in an intermediate regime, namely when the process is observed at a sampling rate $\Delta>0$ fixed, has been studied in van Es \textit{et al.} \cite{VAN-ES07} and in the more general setting of L\'evy processes by Comte and Genon-Catalot \cite{Comte10} and Rei{\ss} \cite{Reiss}. van Es \textit{et al.} \cite{VAN-ES07} provide a consistent kernel density estimator of the compound density of a compound Poisson process of known intensity. They also focus on the nonzero increments for the estimation, but sidestep the problem of the random number of data $N_T$ by assuming that they have a sample of a given size.

The estimator corrected at order $K$ presented here should extend to intermediate regime where $\Delta_T\rightarrow\Delta_\infty<1$ and the rate of convergence given in Theorem \ref{thm Poisson} should generalise in $$\max\big\{T^{-{\alpha(s,p,\pi)}}, \Delta_\infty^{K+1}\big\}.$$ An improvement of the results would be the estimation of the compound density of renewal reward processes, or Continuous Time Random Walk, where it is no longer imposed that the elapsed time between jumps is exponentially distributed. Then the L\'evy property is lost, the increments of the renewal process are no longer independent nor identically distributed. An estimation strategy based on the L\'evy-Kintchine formula is not possible. Such processes enable to model random phenomena where the elapse time between events is not memoryless; they have many applications for instance in finance (see Meerschaert \textit{et al.} \cite{Meerschaert06}), in biology (see Fedotov \textit{et al.} \cite{Fedotov}) or for modelling earthquakes (see Helmstetter \textit{et al.} \cite{Helmstetter}).

\section{Proof of Theorem \ref{thm Poisson}}\label{Section proof}

In the sequel $ \mathfrak{C}$ denotes a generic constant which may vary from line to line. Its dependencies may be indicated in the index.
\subsection{Proof of part 1) of Theorem \ref{thm Poisson}}

\subsubsection*{Preliminary lemmas}

To prove part 1) of Theorem \ref{thm Poisson} we apply the general results of Kerkyacharian and Picard \cite{KP00}. For that we establish some technical lemmas.
\begin{lemma}\label{lem loicontY boule besov}
If $f$ belongs to $\mathcal{F}(s,{\pi},\mathfrak{M})$ then for $m\geq 1$, $\mathbf{P}_\Delta[f]^{\star m}$ also belongs to $\mathcal{F}(s,{\pi},\mathfrak{M})$.
\end{lemma}

\begin{proof}[{Proof of Lemma \ref{lem loicontY boule besov}}]
It is straightforward to derive $\big\|\mathbf{P}_\Delta[f]^{\star m}\big\|_{L_1(\R)}=1$. The remainder of the proof is a consequence of the following result:
Let $f\in \mathcal{B}^s_{{\pi} \infty}$ and $g\in L_1$ we have \begin{align}\|f\star g\|_{s{\pi} \infty}\leq \|f\|_{s{\pi} \infty}\|g\|_{L_1(\R)}.\label{eq etoile}\tag{$\diamondsuit$}\end{align}
To prove the \eqref{eq etoile} we use the following norm which is equivalent to the Besov norm (see \cite{KerkPicTsyb}) \begin{align}\label{eq Besov bis}
\|\nu\|_{s{\pi} \infty}=\|\nu\|_{L_\pi(\R)}+\|\nu^{(n)}\|_{L_\pi(\R)}+\Big\|\frac{w_{\pi}^2(\nu^{(n)},t)}{t^a}\Big\|_\infty
\end{align} where $s=n+a$, $n\in\N$ and $a\in (0,1]$, and $w$ is the modulus of continuity $$w_{\pi}^2(\nu,t)=\underset{|h|\leq t}{\sup}\big\|\mathbf{D} ^h\mathbf{D} ^h[\nu]\big\|_{L_\pi(\R)},$$ where $\mathbf{D} ^h [\nu](x)=\nu(x-h)-\nu(x)$.
The result is a consequence of Young's inequality and elementary properties of the convolution product. We use the definition \eqref{eq Besov bis} of the norm and treat each term separately. First Young's inequality gives \begin{align}\label{eq ine1}\|f_1\star f_2\|_{L_\pi(\R)}\leq \|f_1\|_{L_\pi(\R)}\|f_2\|_{L_1(\R)}.\end{align} Then the differentiation property of the convolution product leads for $n\geq1$ to \begin{align}\label{eq ine2}\Big\|\frac{d^n}{dx^n}(f_1\star f_2)\Big\|_{L_\pi(\R)}=\Big\|\Big(\frac{d^n}{dx^n} f_1\Big)\star f_2\Big\|_{L_\pi(\R)}\leq \Big\|\frac{d^n}{dx^n} f_1\Big\|_{L_\pi(\R)}\|f_2\|_{L_1(\R)}.\end{align} Finally translation invariance of the convolution product enables to get \begin{align}\big\|\mathbf{D} ^h\mathbf{D} ^h[(f_1\star f_2)^{(n)}]\big\|_{L_\pi(\R)}&=\big\|(\mathbf{D} ^h\mathbf{D} ^h[f_1^{(n)}])\star f_2\big\|_{L_\pi(\R)}\nonumber\\&\leq \big\|\mathbf{D} ^h\mathbf{D} ^h[f_1^{(n)}]\big\|_{L_\pi(\R)}\|f_2\|_{L_1(\R)}.\label{eq ine3}\end{align} Inequality \eqref{eq etoile} is then obtained by bounding \eqref{eq Besov bis} with \eqref{eq ine1}, \eqref{eq ine2} and \eqref{eq ine3} lead to the result.
 \noindent To complete the proof of Lemma \ref{lem loicontY boule besov}, we apply $m-1$ times \eqref{eq etoile} which leads to $$\forall m\in \N\setminus\{0\},\ \ \ \ \ \big\|\mathbf{P}_\Delta[f]^{\star m}\big\|_{s{\pi}\infty}\leq \big\|\mathbf{P}_\Delta[f]\big\|_{s{\pi}\infty}.$$ The triangle inequality gives $\|\mathbf{P}_\Delta[f]^{\star m}\|_{s{\pi}\infty}\leq \|f\|_{s{\pi}\infty}\leq\mathfrak{M}$ which concludes the proof.
\end{proof}

\begin{lemma}\label{lem RosConvol Poisson}
Let $2^{j}\leq  N_T$ then for all $m\in \N\setminus\{0\}$ and for $p\geq 1$ we have $$\E\big[\big|\widehat{\gamma}_{jk}^{(m)}-\gamma_{jk}^{(m)}\big|^p\big|N_T\big]\leq \mathfrak{C}_{p,m,\|g\|_{L_p(\R)},\mathfrak{M},\vartheta} N_T^{-p/2},$$
where $\widehat{\gamma}_{jk}^{(m)}$ is defined in \eqref{eq est Coeffconvol} and \begin{align}\label{eq coeff gamma}\gamma_{jk}^{(m)}=\int g_{jk}(y) \mathbf{P}_\Delta[f]^{\star m}(y)dy.\end{align}
\end{lemma}

\begin{proof}[{Proof of Lemma \ref{lem RosConvol Poisson}}]
The proof is obtained with Rosenthal's inequality: let $p\geq 1$ and let $(Y_1,\ldots,Y_n)$ be independent random variables such that $\E[Y_i]=0$ and $\E\big[|Y_i|^p\big]<\infty$. Then there exists $\mathfrak{C}_p$ such that \begin{align}\label{eq Ros}\E\bigg[\Big|\sum_{i=1}^nY_i \Big|^p\bigg]\leq \mathfrak{C}_p\bigg\{\sum_{i=1}^n\E\big[|Y_i|^p\big]+\Big(\sum_{i=1}^n\E\big[|Y_i|^2\big]\Big)^{p/2}\bigg\}.\end{align}

\noindent The $\big(\mathbf{D}^{\Delta_T}_m X_{S_i}\big)$ are independent and identically distributed with common density $\mathbf{P}_{\Delta_T}[f]^{\star m}$ and $\E [\widehat{\gamma}_{jk}^{(m)}]=\gamma_{jk}^{(m)}.$ Then $ \widehat{\gamma}_{jk}^{(m)}-\gamma_{jk}^{(m)}$ is a sum of $N_{T,m}=\lfloor N_T/m\rfloor$ centered, independent and identically distributed random variables. It follows that
\begin{align*}
\E \big[\big|g_{jk}(\mathbf{D}^{\Delta_T}_m X_{S_i})\big|^p\big]&\leq
2^p2^{jp/2}\int |g(2^jy-k)|^p  \mathbf{P}_{\Delta_T}[f]^{\star m}(y)dy\\&=
2^p2^{j(p/2-1)}\int |g(z)|^p  \mathbf{P}_{\Delta_T}[f]^{\star m}\Big(\frac{z+k}{2^j}\Big)dz\\
&\leq 2^p2^{j(p/2-1)}\|g\|_{L_p(\R)}^p\big\|\mathbf{P}_{\Delta_T}[f]^{\star m}\big\|_\infty,\end{align*} where we made the substitution $z=2^jx-k$. To control $\|\mathbf{P}_{\Delta_T}[f]^{\star m}\|_\infty$ we use the Sobolev embeddings (see \cite{Cohen,Donoho96,KerkPicTsyb}) \begin{align}\label{eq Sobolev embed}\mathcal{B}^s_{{\pi} \infty}\hookrightarrow \mathcal{B}^{s'}_{p \infty}\ \ \ \mbox{ and }\ \ \  \mathcal{B}^{s'}_{{\pi} \infty}\hookrightarrow\mathcal{B}^s_{{\infty} \infty},\end{align}where $p>\pi$, $s\pi>1$ and $s'=s-1/\pi+1/p$, it follows that $$\|\mathbf{P}_{\Delta_T}[f]^{\star m}\|_\infty\leq \mathfrak{C}_{s,\pi} \|\mathbf{P}_{\Delta_T}[f]^{\star m}\|_{\mathcal{B}^s_{\pi\infty(\mathcal{D})}}.$$ We deduce from Lemma \ref{lem loicontY boule besov} that $\|\mathbf{P}_{\Delta_T}[f]^{\star m}\|_\infty\leq \mathfrak{C}_{s,\pi} \mathfrak{M}$. We get $$\E \big[\big|g_{jk}(\mathbf{D}^{\Delta_T}_m X_{S_i})\big|^p\big]\leq 2^p2^{j(p/2-1)}\|g\|_{L_p(\R)}^p\mathfrak{M}$$ and $\E \big[\big|g_{jk}(\mathbf{D}^{\Delta_T}_m X_{S_i})\big|^2\big]\leq \mathfrak{M}$ since $\|g\|_2^2=1$.

The accept-reject algorithm ensures that for all $n\geq1$ the increments $\big(\mathbf{D}^{\Delta_T} X_{S_1},\ldots,\mathbf{D}^{\Delta_T} X_{S_n}\big)$ are independent of $N_T$ and then $N_{T,m}$. Indeed the $\big(\mathbf{D}^{\Delta_T}X_i,i=1,\ldots,\lfloor T\Delta_T^{-1}\rfloor\big)$ are independent and identically distributed and the $\big(\mathbf{D}^{\Delta_T} X_{S_i}\big)$ are constructed with $S_i=\inf\big\{j>S_{i-1},\mathbf{D}^{\Delta_T}X_j\ne0\big\}.$ Therefore we can apply Rosenthal's inequality conditional on $N_T$ to $ \widehat{\gamma}_{jk}^{(m)}-\gamma_{jk}^{(m)}$ and derive for $p\geq 1$ \begin{align*}
 \E\big[\big| \widehat{\gamma}_{jk}^{(m)}-\gamma_{jk}^{(m)}\big|^p\big|N_T\big]&\leq \mathfrak{C}_p \Big\{2^p\Big(\frac{2^{j}}{ N_{T,m}}\Big)^{p/2-1}\|g\|_{L_p(\R)}^p\mathfrak{M}+\mathfrak{M}^{p/2}\Big\}N_{T,m}^{-p/2}.
 \end{align*} This concludes the proof.\end{proof}

\begin{lemma}\label{lem BernConvol Poisson} Choose $j$ and $c$ such that $$2^jN_T^{-1}\log\big(
N_T^{1/2}\big)\leq 1\mbox{ and }c^2\geq \frac{16m}{3}\Big(\mathfrak{M}+\frac{c\|g\|_\infty}{6}\Big).$$ For all $m\in \N\setminus\{0\}$ and $r\geq 1$ let $\kappa_r=c r$. We have $$\PP\Big(|\widehat{\gamma}_{jk}^{(m)}-\gamma_{jk}^{(m)}|\geq \frac{\kappa_r}{2}
N_T^{-1/2}\sqrt{\log\big(N_T^{1/2}\big)}\Big|N_T\Big)\leq N_T^{-r/2},$$
where $\widehat{\gamma}_{jk}^{(m)}$ is defined in \eqref{eq est Coeffconvol} and $\gamma_{jk}^{(m)}$ in \eqref{eq coeff gamma}.
\end{lemma}

\begin{proof}[Proof of Lemma \ref{lem BernConvol Poisson}]
The proof is obtained with Bernstein's inequality. Consider $Y_1,\ldots,Y_n$ independent random variables such that $|Y_i|\leq \mathfrak{A}$, $\E[Y_i]=0$ and $b_n^2=\sum_{i=1}^n\E[Y_i^2]$. Then for any $\lambda>0$, \begin{align}\label{eq Bernstein}\PP\Big(\Big|\sum_{i=1}^nY_i\Big| >\lambda\Big)\leq 2\exp\Big(-\frac{\lambda^2}{2(b_n^2+\frac{\lambda \mathfrak{A}}{3})}\Big).\end{align}

\noindent For all $m\geq1$, $ \widehat{\gamma}_{jk}^{(m)}-\gamma_{jk}^{(m)}$ is a sum of $N_{T,m}=\lfloor N_T/m\rfloor$ centered independent and identically distributed random variables bounded by $2^{j/2}\|g\|_\infty$ and $\E \big[\big|g_{jk}(\mathbf{D}^{\Delta_T}_m X_{S_i})\big|^2\big]\leq\mathfrak{ M}.$ The accept-reject algorithm ensures that for all $n\geq1$ the increments $\big(\mathbf{D}^{\Delta_T} X_{S_1},\ldots,\mathbf{D}^{\Delta_T} X_{S_n}\big)$ are independent of $N_T$ (see proof of Lemma \ref{lem RosConvol Poisson}), we apply Bernstein's inequality conditional on $N_T$. We have
 \begin{align*}
 \PP\Big(|\widehat{\gamma}_{jk}^{(m)}&-\gamma_{jk}^{(m)}|\geq \frac{\kappa_r}{2}N_T^{-1/2}\sqrt{\log\big(N_T^{1/2}\big)}\Big|N_T\Big)\\
 \leq&2\exp\Bigg(-\frac{\kappa_r^2N_T^{-1}\log\big(N_T^{1/2}\big)N_{T,m}^2} {8\Big(N_{T,m}\mathfrak{M}+\frac{\kappa_rN_{T,m}N_T^{-1/2}\sqrt{\log\big(N_T^{1/2}\big)}2^{j/2}\|g\|_\infty}{6}\Big)}\Bigg)\\
 =&2\exp\Bigg(-\frac{c^2rN_T^{-1}N_{T,m}} {8\Big(\mathfrak{M}+\frac{\kappa_rN_T^{-1/2} \sqrt{\log\big(N_T^{1/2}\big)}2^{j/2}\|g\|_\infty}{6}\Big)}r\log\big(N_T^{1/2}\big)\Bigg).
 \end{align*}
Using that $$mN_T^{-1}N_{T,m}=\frac{m}{N_T}\Big\lfloor \frac{N_T}{m}\Big\rfloor \geq \frac{3}{2},$$ for $T$ large enough and $2^{j/2} N_T^{-1} \sqrt{\log\big(N_T^{1/2}\big)}\leq 1$ it follows that \begin{align*}
 \PP\Big(|\widehat{\gamma}_{jk}^{(m)}-\gamma_{jk}^{(m)}&|\geq \frac{\kappa_r}{2}N_T^{-1/2}\sqrt{\log\big(N_T^{1/2}\big)}\Big|N_T\Big)\nonumber
\\& \leq2\exp\Bigg(-\frac{3c^2r} {16m\big(\mathfrak{M}+\frac{\kappa_r\|g\|_\infty}{6}\big)}r\log\big(N_T^{1/2}\big)\Bigg).\end{align*} With $c^2\geq \frac{16m}{3}\big(\mathfrak{M}+\frac{c\|g\|_\infty}{6}\big)$ we get \begin{align*}
 \PP\Big(|\widehat{\gamma}_{jk}^{(m)}-\gamma_{jk}^{(m)}&|\geq \frac{\kappa_r}{2}N_T^{-1/2}\sqrt{\log\big(N_T^{1/2}\big)}\Big|N_T\Big)\nonumber
 \leq N_T^{-r/2}.\end{align*}The proof is complete.\end{proof}

\subsubsection*{Completion of proof of part 1) of Theorem \ref{thm Poisson}}
Part 1) of Theorem \ref{thm Poisson} is a consequence of Lemma \ref{lem loicontY boule besov}, \ref{lem RosConvol Poisson}, \ref{lem BernConvol Poisson} and of the general theory of wavelet threshold estimators of \cite{KP00}. It suffices to have conditions (5.1) and (5.2) of Theorem 5.1 of \cite{KP00}, which are satisfied --Lemma \ref{lem RosConvol Poisson} and \ref{lem BernConvol Poisson}-- with $c(T)=N_T^{-1/2}$ and $\Lambda_n=c(T)^{-1}$ (with the notations of \cite{KP00}). We can now apply Theorem 5.1, its Corollary 5.1 and Theorem 6.1 of \cite{KP00} to obtain the result.

\subsection{Proof of part 2) of Theorem \ref{thm Poisson}}

\subsubsection*{Preliminary result}

The result of part 1) of Theorem \ref{thm Poisson} where given conditional on $N_T$. To prove part 2) we replace $N_T$ by its deterministic counterpart. We introduce the following result.
 \begin{proposition}\label{prop control rate}
For all $r>0$, there exist $1\leq\mathfrak{C}_\vartheta<\infty,$ where $\vartheta\rightarrow \mathfrak{C}_\vartheta$ is continuous, such that \begin{align*}
1/\mathfrak{C}_\vartheta T^{-r}\leq\E\big[N_T^{-r}\big]&\leq  \mathfrak{C}_\vartheta  T^{-r}.
\end{align*}
\end{proposition}

\begin{proof}[Proof of Proposition \ref{prop control rate}]
We have $$N_T=\sum_{i=1}^{\lfloor T\Delta_T^{-1}\rfloor}\mathds{1}_{\{\mathbf{D}^{\Delta_T} X_i\ne0\}},$$ where $$\E\big[\mathds{1}_{\{\mathbf{D}^{\Delta_T} X_i\ne0\}}\big]=p(\Delta_T)=1-\exp(-\vartheta\Delta_T).$$
Introduce $Y_i=\mathds{1}_{\{\mathbf{D}^{\Delta_T} X_i\ne0\}}-p(\Delta_T)$, the $Y_i$ are centered independent and identically distributed random variables bounded by 2 and $\E [Y_i^2]\leq p(\Delta_T)$, it follows from Bernstein's inequality \eqref{eq Bernstein} that for $\lambda>0$
\begin{align}\label{eq N control}
\PP\Big(\Big|\frac{N_T}{\lfloor T\Delta_T^{-1}\rfloor}-p(\Delta_T)\Big|>\lambda\Big)
&\leq\exp\bigg(- \frac{\lfloor T\Delta_T^{-1}\rfloor\lambda^2}{2\big( p(\Delta_T)+\frac{2\lambda}{3}\big)}\bigg).
\end{align} We choose $\lambda=p(\Delta_T)/2$, on the set $\big\{\big|\frac{N_T}{\lfloor T\Delta_T^{-1}\rfloor}-p(\Delta_T)\big|\leq \lambda\big\}$ we have $$\lfloor T\Delta_T^{-1}\rfloor\frac{p(\Delta_T)}{2}\leq N_T\leq \lfloor T\Delta_T^{-1}\rfloor\frac{3p(\Delta_T)}{2}.$$ Moreover for $\Delta_T$ small enough we have that $$\frac{\vartheta}{2}\leq p(\Delta_T)=1-\exp(-\vartheta\Delta_T)\leq \vartheta\Delta_T.$$ We have for all $\lambda>0$
\begin{align*}
\E\big[N_T^{-r}\big]&=\E\Big[N_T^{-r} \mathds{1}_{\big\{|\frac{N_T}{\lfloor T\Delta_T^{-1}\rfloor}-p(\Delta_T)|>\lambda\big\}}\Big]+\E\Big[N_T^{-r}\mathds{1}_{\big\{|\frac{N_T}{\lfloor T\Delta_T^{-1}\rfloor}-p(\Delta_T)|\leq \lambda\big\}}\Big].
\end{align*} Since for $r>0$ the function $x\rightarrow x^{-r}$ is decreasing and $N_T\geq1$ we have using \eqref{eq N control} the upper bound\begin{align*}
\E\big[N_T^{-r}\big]&\leq \PP\Big(\Big|\frac{N_T}{\lfloor T\Delta_T^{-1}\rfloor}-p(\Delta_T)\Big|>\frac{p(\Delta_T)}{2}\Big)+ \Big(\frac{\lfloor T\Delta_T^{-1}\rfloor p(\Delta_T)}{2}\Big)^{-r}\\
&\leq \exp\bigg(- \frac{\lfloor T\Delta_T^{-1}\rfloor p(\Delta_T)^2}{8\big( p(\Delta_T)+\frac{p(\Delta_T)}{3}\big)}\bigg)+ \Big(\frac{\lfloor T\Delta_T^{-1}\rfloor p(\Delta_T)}{2}\Big)^{-r}\\
&\leq \exp\bigg(- \frac{3\vartheta}{64}T\bigg)+ \Big(\frac{T\vartheta}{4}\Big)^{-r}.\end{align*} For the lower bound we have
\begin{align*}
\E\big[N_T^{-r}\big]&\geq \Big(\frac{3\lfloor T\Delta_T^{-1}\rfloor p(\Delta_T)}{2}\Big)^{-r}\geq
\Big(\frac{3T\vartheta}{2}\Big)^{-r}.\end{align*} Then there exists $1\leq \mathfrak{C}_\vartheta<\infty$ with $\vartheta\rightarrow \mathfrak{C}_\vartheta$ continuous such that $$1/\mathfrak{C}_\vartheta T^{-r}\leq\E\big[N_T^{-r}\big]\leq  \mathfrak{C}_\vartheta  T^{-r}.$$ The proof is now complete.\end{proof}

\subsubsection*{Completion of proof of part 2) of Theorem \ref{thm Poisson}}

To prove Theorem \ref{thm Poisson} we define the quantity for $K$ in $\N$ and $x$ in $\mathcal{D}$
 \begin{align*}
\widehat{f}^K_{T,\Delta}(x)&=\sum_{m=1}^{K+1}\frac{(-1)^{m+1}}{m}\frac{\big(e^{\vartheta\Delta}-1\big)^m}{ \vartheta\Delta}\widehat{P_{\Delta,m}}(x).
 \end{align*} It is the estimator of $f$ one would compute if $\vartheta$ were known. We decompose the $L_p$ error as follows \begin{align*}
\big(\E\big[\|\widetilde{f}^K_{T,\Delta_T}-f\|_{L_p(\mathcal{D})}^p\big]\big)^{1/p}\leq& \big(\E\big [\|\widetilde{f}^K_{T,\Delta_T}-\widehat{f}^K_{T,\Delta_T}\|_{L_p(\mathcal{D})}^p\big]\big)^{1/p}\\ &+\big(\E\big[\|\widehat{f}^K_{T,\Delta_T}-f\|_{L_p(\mathcal{D})}^p\big]\big)^{1/p},
\end{align*} and control each term separately.

First we control $\E\big[\|\widehat{f}^K_{T,\Delta_T}-f\|_{L_p(\mathcal{D})}^p\big]$, using the triangle inequality we get
\begin{align}
&\Big(\E\Big[\Big\|\sum_{m=1}^{K+1} \frac{(-1)^{m+1}}{m}\frac{\big(e^{\vartheta{\Delta_T}} -1\big)^m}{\vartheta{\Delta_T}} \widehat{P_{\Delta_T,m}}-\mathbf{P}_{\Delta_T}^{-1}\big[\mathbf{P}_{\Delta_T}[f]\big]\Big\|_{L_p(\mathcal{D})}^p \Big]\Big)^{1/p}\nonumber \\
&\leq\sum_{m=1}^{K+1}\frac{\big(e^{\vartheta{\Delta_T}}-1\big)^{m}}{m\vartheta{\Delta_T}} \big(\E\big[\big\|\widehat{P_{\Delta_T,m}}-\mathbf{P}_{\Delta_T}[f]^{\star m}\big\|_{L_p(\mathcal{D})}^p\big]\big)^{1/p}\label{eq thm T11}\\
&\ \ \ +\sum_{m=K+2}^{\infty}\frac{\big(e^{\vartheta{\Delta_T}}-1\big)^m}{m\vartheta{\Delta_T}} \|\mathbf{P}_{\Delta_T}[f]^{\star m}\|_{L_p(\R)}\label{eq thm T12}.
\end{align}
To bound \eqref{eq thm T11} we use part 1) of Theorem \ref{thm Poisson} in which the supremum is taken over the class $\{\mathbf{P}_{\Delta_T}[f]^{\star m}\in\mathcal{F}(s,{\pi},\mathfrak{M})\}$. With the inclusion \begin{align*}\{\mathbf{P}_{\Delta_T}[f]^{\star m},f\in\mathcal{F}(s,{\pi},\mathfrak{M})\}\subset\mathcal{F}(s,{\pi},\mathfrak{M})\end{align*} and Proposition \ref{prop control rate} applied with $r={\alpha(s,p,\pi)} p>0,$ we deduce the upper bound for $m\geq1$ \begin{align}\E\big[\big\|\widehat{P_{\Delta_T,m}}-\mathbf{P}_{\Delta_T}^{-1}\big[\mathbf{P}_{\Delta_T}[f]\big\|_{L_p(\mathcal{D})}^p\big]&\leq \mathfrak{C}\E\big[N_T^{-{\alpha(s,p,\pi)} p}\big]\nonumber\\&\leq\mathfrak{C} T^{-{\alpha(s,p,\pi)} p}\label{eq bound T11} ,\end{align} where $\mathfrak{C}$ depends on $(s,\pi,p,\mathfrak{M},\phi,\psi,K,\vartheta)$. To bound \eqref{eq thm T12} Young's inequality and $\big\|\mathbf{P}_{\Delta_T}[f]\big\|_{L_1(\R)}=1$ enable to get $$\big\|\mathbf{P}_{\Delta_T}[f]^{\star m}\big\|_{L_p(\R)}\leq\big\|\mathbf{P}_{\Delta_T}[f]\big\|_{L_p(\R)}\ \ \ \mbox{ for }m\geq1.$$ The triangle inequality leads to $\big\|\mathbf{P}_{\Delta_T}[f]\big\|_{L_p(\R)}\leq\|f\|_{L_p(\R)}$ and we use the Sobolev embeddings \eqref{eq Sobolev embed} to get $\|f\|_{L_p(\R)}\leq \mathfrak{C}_{s,\pi,p}\mathfrak{M}$. We derive the upper bound \begin{align}\sum_{m=K+2}^{\infty}\frac{1}{m}\frac{\big(e^{\vartheta{\Delta_T}}-1\big)^m}{\vartheta{\Delta_T}} &\big\|\mathbf{P}_{\Delta_T}[f]^{\star m}\big\|_{L_p(\R)}\nonumber\\&\leq\|f\|_{L_p(\R)}\sum_{m=K+2}^{\infty}\frac{1}{m}\frac{\big(e^{\vartheta{\Delta_T}}-1\big)^m}{\vartheta{\Delta_T}}\nonumber\\&\leq \mathfrak{C}_{K,\vartheta,\mathfrak{M}} {\Delta_T}^{K+1}.\label{eq bound T12}\end{align} Thus from \eqref{eq bound T11} and \eqref{eq bound T12} we obtain \begin{align*}
\underset{f\in\mathcal{F}(s,{\pi},\mathfrak{M})}{\sup}\big(\E \big[\big\|\widehat{f}^K_{T,\Delta_T} -f\big\|_{L_p(\mathcal{D})}^p\big] \big)^{1/p}&\leq \mathfrak{C} \max\big\{T^{-{\alpha(s,p,\pi)}},\Delta_T^{K+1}\big\},
\end{align*} where $\mathfrak{C}$ depends on $(s,\pi,p,\mathfrak{M},\phi,\psi,K,\vartheta)$. Since $\vartheta\rightarrow \mathfrak{C}$ is continuous we get for $p\geq 1$ \begin{align*}\underset{\vartheta\in [\underline{\mathfrak{T}},\overline{\mathfrak{T}}] }{\sup}\ \underset{f\in\mathcal{F}(s,{\pi},\mathfrak{M})}{\sup}\big(\E \big[\big\|\widehat{f}^K_{T,\Delta_T} -f\big\|_{L_p(\mathcal{D})}^p\big] \big)^{1/p}&\leq \mathfrak{C}_{K,\mathfrak{M}} \max\big\{T^{-{\alpha(s,p,\pi)}},\Delta_T^{K+1}\big\},\end{align*} where $\mathfrak{C}$ depends on $(s,\pi,p,\mathfrak{M},\phi,\psi,K)$

\

We now control $\E\big [\|\widetilde{f}^K_{T,\Delta_T}-\widehat{f}^K_{T,\Delta_T}\|_{L_p(\mathcal{D})}^p\big] $ and use \eqref{eq est tau0} to derive
\begin{align*}
\widetilde{f}^K_{T,\Delta_T}
=\sum_{m=1}^{K+1}\frac{(-1)^{m}}{m} \frac{\big((1-\widehat{p}_T)^{-1}-1\big)^m}{\log(1-\widehat{p}_T)}\widehat{P_{\Delta_T,m}},\end{align*} where $\widehat{P_{\Delta_T,m}}$ does not depend on $\vartheta$ (see \eqref{eq est Coeffconvol}). Define $$G_m(x)=\frac{((1-x)^{-1}-1)^m}{\log(1-x)}.$$ The triangle inequality leads to
\begin{align*}
\big(\E\big[\|\widehat{f}^K_{T,\Delta_T}\big(\widehat{\vartheta}\big)-&\widehat{f}^K_{T,\Delta_T}\|_{L_p(\mathcal{D})}^p \big]\big)^{1/p}\\&\leq \sum_{m=1}^{K+1}\big(\E\big[\|\big(G_m(\widehat{p}_T)-G_m(p(\Delta_T))\big)\widehat{P_{\Delta_T,m}}\|_{L_p(\mathcal{D})}^p \big]\big)^{1/p},\end{align*} where $p(\Delta_T)$ verifies $p(\Delta_T)=1-e^{-\vartheta\Delta_T}\leq\mathfrak{C}_{\underline{\mathfrak{T}},\overline{\mathfrak{T}}}\Delta_T$ since $$0<1-e^{-\underline{\mathfrak{T}}\Delta_T}\leq 1-e^{-\vartheta\Delta_T}\leq 1-e^{-\mathfrak{T}\Delta_T}<1.$$ Moreover, we have $$G'_m(x)=\frac{mx^{m-1}}{(1-x)^{m+1}\log(1-x)}+\frac{x^{m}}{(1-x)^{m+1}\big(\log(1-x)\big)^2},$$ then for all $m\geq1$ $xG'_m(x)$ is continuous over $(0,1/2]$ and converges to 0 when $x\rightarrow 0.$ We deduce \begin{align*}
\E\big[\|\widehat{f}^K_{T,\Delta_T}\big(\widehat{\vartheta}\big)-&\widehat{f}^K_{T,\Delta_T}\|_{L_p(\mathcal{D})}\big]^{1/p}\\&\leq \mathfrak{C}_{\underline{\mathfrak{T}},\overline{\mathfrak{T}},K} \Delta_T^{-1} \E\big[\big\|\big(\widehat{p}_T-p(\Delta_T)\big)\widehat{P_{\Delta_T,m}}\big\|_{L_p(\mathcal{D})}^p\big]^{1/p}.\end{align*} Cauchy-Schwarz inequality leads to \begin{align*}\E\big[\big\|\big(\widehat{p}_T-p(\Delta_T)\big)&\widehat{P_{\Delta_T,m}}\big\|_{L_p(\mathcal{D})}^p\big]^2\\&\leq \E\Big[\big\|\big(\widehat{p}_T-p(\Delta_T)\big)\big\|_{2p}^{2p}\Big] \E\Big[\big\|\widehat{P_{\Delta_T,m}}\big\|_{L_{2p}(\mathcal{D})}^{2p}\Big] ,\end{align*} where using part 1) of Theorem \ref{thm Poisson} and that $N_T\geq1$ we have
\begin{align}\E\Big[\big\|\widehat{P_{\Delta_T,m}}\big\|_{L_{2p}(\mathcal{D})}^{2p}\Big]& \leq \E\Big[\|\widehat{P_{\Delta_T,m}}-\mathbf{P}_{\Delta_T}[f]^{\star m}\|_{L_{2p}(\mathcal{D})}^{2p}\Big]+\|\mathbf{P}_{\Delta_T}[f]^{\star m}\|_{L_{2p}(\mathcal{D})}^{2p}\nonumber\\&\leq \mathfrak{C}\E[N_T^{-2{\alpha(s,p,\pi)} p}]+\mathfrak{M}^{2p}\nonumber\\&\leq \mathfrak{C}\label{eq thm bound}\end{align} where $\mathfrak{C}$ depends on $(s,\pi,p,\mathfrak{M},\phi,\psi)$. We apply Rosenthal's inequality \eqref{eq Ros} to conclude the proof: $\widehat{p}_T-p(\Delta_T)$ is the sum of independent and identically distributed centered random variables $$\big(Y_i=\mathds{1}_{\{\mathbf{D}^{\Delta_T}X_i\ne 0\}}-p(\Delta_T),i\in \{1,\ldots,\lfloor T\Delta_T^{-1}\rfloor\}\big)$$ where $\E[|Y_i|^{2p}]\leq 2^{2p}\E\big[\mathds{1}_{\{\mathbf{D}^{\Delta_T}X_i\ne0\}}^{2p}\big]\leq \mathfrak{C}_{p,\mathfrak{T}}\Delta_T$ and $\E[|Y_i|^2]\leq\mathfrak{C}_{\underline{\mathfrak{T}},\overline{\mathfrak{T}}}\Delta_T$. Rosenthal's inequality \eqref{eq Ros} gives \begin{align}\E\big[\|\widehat{p}_T-p(\Delta_T)&\|_{2p}^{2p}\big]\nonumber\\&\leq \mathfrak{C}_{p,\underline{\mathfrak{T}},\overline{\mathfrak{T}}}\lfloor T\Delta_T^{-1}\rfloor^{-{2p}}\big(\lfloor T\Delta_T^{-1}\rfloor\Delta_T+(\lfloor T\Delta_T^{-1}\rfloor\Delta_T)^{p}\big).\label{eq thm Tros}\end{align} It follows from \eqref{eq thm bound} and \eqref{eq thm Tros} that
\begin{align*}
\E\big[\|\widehat{f}^K_{T,\Delta_T}\big(\widehat{\vartheta}\big)-\widehat{f}^K_{T,\Delta_T}&\|_{L_p(\mathcal{D})}\big]^{1/p}\\&\leq \mathfrak{C} \Delta_T^{-1}\lfloor T\Delta_T^{-1}\rfloor^{-1}\big(T^{1/(2p)}+T^{1/2}\big),\end{align*} where $\mathfrak{C}$ depends on $(s,\pi,p,\mathfrak{M},\phi,\psi,\underline{\mathfrak{T}},\overline{\mathfrak{T}},K)$. We deduce for $p\geq 1$ \begin{align*}\underset{\vartheta\in [\underline{\mathfrak{T}},\overline{\mathfrak{T}}] }{\sup}\underset{f\in\mathcal{F}(s,{\pi},\mathfrak{M})}{\sup}\big(\E\big[\|\widehat{f}^K_{T,\Delta_T} \big(\widehat{\vartheta}\big)-\widehat{f}^K_{T,\Delta_T}&\|_{L_p(\mathcal{D})}^p\big]\big)^{1/p}\\&\leq \mathfrak{C}\big(T^{-(1-1/(2p))}+T^{-1/2}\big)\end{align*} where $\mathfrak{C}$ depends on $(s,\pi,p,\mathfrak{M},\phi,\psi,\underline{\mathfrak{T}},\overline{\mathfrak{T}},K)$ and which is negligible compared to $T^{-{\alpha(s,p,\pi)}}$ since ${\alpha(s,p,\pi)}\leq1/2$. The proof of Theorem \ref{thm Poisson} is now complete.

\section{Appendix}

\subsection{Proof of Proposition \ref{PropDefOperator}}

Let $x\in\R$, we have by stationarity of the increments of the process $X$
\begin{align*}
\PP(\mathbf{D}^{\Delta} X_{S_1}\leq x)&=\PP(X_\Delta\leq x|X_\Delta\ne0)\\&=\sum_{m=0}^\infty\PP(X_\Delta\leq x|R_\Delta=m,R_\Delta\ne0) \PP(R_\Delta=m) \\
&=\sum_{m=1}^\infty p_m(\Delta) \PP(X_\Delta\leq x|R_\Delta=m) \end{align*} where
$
\PP(X_\Delta\leq x|R_\Delta=m)=\int_{-\infty}^xf^{\star m}(y)dy$ for $m\geq1$. It follows
\begin{align*}
\PP(\mathbf{D}^{\Delta} X_{S_1}\leq x)&=\int_{-\infty}^x \mathbf{P}_\Delta[f](y)dy.
\end{align*}
Immediate computation give the expression of $p_m(\Delta)$. For the control of $p_1(\Delta)$ the assertion $p_1(\Delta)\leq1$ is immediate since $p_1(\Delta)$ is a probability. Moreover we have $$\exp\big(\vartheta\Delta\big)-1=\vartheta\Delta\Big(1+\vartheta\Delta\sum_{m=2}^\infty \frac{\big(\vartheta\Delta\big)^{m-2}}{m!}\Big),$$ where $$g(\Delta):=\sum_{m=2}^\infty \frac{\big(\vartheta\Delta\big)^{m-2}}{m!}=\frac{1}{\big(\vartheta\Delta\big)^2}\big(\exp\big(\vartheta\Delta\big)-1-\vartheta\Delta\big)\longrightarrow \frac{1}{2}\ \ \ \mbox{ as } \Delta\rightarrow0.$$ Since $g$ is continuous, there exists $\Delta_0>0$ such that for all $\Delta\leq\Delta_0$ we have $g(\Delta)\leq1.$ It follows for $\Delta\leq\Delta_0$ that \begin{align*}p_1(\Delta)&\geq \frac{1}{1+\vartheta\Delta}\geq1-\vartheta\Delta.
\end{align*}

\subsection{Proof of Lemma \ref{lem Inverse T}}

Let $\mathbf{F}[f]$ denote the Fourier transform of $f$ and take $h$ such that $h=\mathbf{P}_\Delta[f]$. Using the one-to-one mapping between densities and their Fourier transform we show the relation for the Fourier transforms. The linearity of the Fourier transform and the relation $\mathbf{F}[f\star g]=\mathbf{F}[f]\mathbf{F}[g]$ give
\begin{align*}\mathbf{F}[h]&=\mathbf{F}\big[\mathbf{P}_\Delta[f]\big]=\frac{1}{e^{\vartheta\Delta}-1}\sum_{m=1}^\infty \frac{\big(\vartheta\Delta\big)^m} {m!}\mathbf{F}[f]^m=\frac{\big(\exp(\vartheta\Delta\mathbf{F}[f])-1\big)}{e^{\vartheta\Delta}-1},\end{align*}
from which we deduce \begin{align*}\mathbf{F}[f]&=\frac{\log\big(1+(e^{\vartheta\Delta}-1)\mathbf{F}[h]\big)}{\vartheta\Delta} =\sum_{m=1}^\infty \frac{(-1)^{m+1}}{m}\frac{(e^{\vartheta\Delta}-1)^m}{\vartheta\Delta}\mathbf{F}[h]^m
\end{align*}
as $\big\|(e^{\vartheta\Delta}-1)\mathbf{F}[h]\big\|_\infty<\big\|e^{\vartheta\Delta}-1\big\|_\infty<1$ holds for $\Delta\leq \log 2$. We take the inverse Fourier transform of the equality to obtain the result.

\section*{Acknowledgements}
This work is a part of the author's Ph.D thesis under the supervision of Marc Hoffmann whom I would like to thanks for his valuable remarks on this paper. The author's research is supported by a PhD GIS Grant.


\begin{thebibliography}{99}
\small{

\bibitem{Lacour} M. Bec, C. Lacour, Adaptive kernel estimation of the L\'evy  density, Hal preprint 0058322 (2011).
\bibitem{Buchmann} B. Buchmann, R. Gr\"ubel, Decompounding: an estimation
problem for Poisson random sums, The Annals of Statistics 31 (2003) 1054--1074.

\bibitem{Cohen} A. Cohen, Numerical Analysis of wavelet methods, Elsevier, 2003.
\bibitem{Comte09} F. Comte, V. Genon-Catalot, Nonparametric estimation for pure jump L\'evy processes
based on high frequency data, Stochastic Process.
Appl. 119 (2009) 4088--4123.
\bibitem{Comte10} F. Comte, V. Genon-Catalot, Nonparametric adaptive estimation for pure jump L\'evy processes, Annales de l'I.H.P., Probability and Statistics 46 (2010) 595--617.
\bibitem{Comte11} F. Comte, V. and Genon-Catalot, Estimation for L\'evy processes from high frequency data within a long time interval, The Annals of Statistics 39 (2011) 803--837.
\bibitem{Donoho96} D.L. Donoho, I.M. Johnstone, G. Kerkyacharian, D. Picard, D, Density estimation by wavelet Thresholding, The Annals of Statistics 24 (1996) 508--539.

\bibitem{Embrechts} P. Embrechts, C. Kl\"{u}ppelberg, M. Mikosch, Modelling Extremal Events, Springer, 1997.
\bibitem{Fedotov} S. Fedotov, A. Iomin, Probabilistic approach to a proliferation and migration dichotomy in the tumor cell
invasion, Arxiv preprint 0711.1304v2 (2008).

\bibitem{Lopez} J.E. Figueroa-L\'opez, C. Houdr\'e, Risk bounds for the nonparametric estimation of L\'evy processes, IMS Lecture
Notes-Monogr. Ser. High dimensional probability 51 (2006) 96--116.
\bibitem{KerkPicTsyb}  W. H\"{a}rdle, G. Kerkyacharian, D. Picard, A. Tsybakov, Wavelets, approximation, and statistical applications,  Springer-Verlag, New York, 1998.
\bibitem{Helmstetter} A. Helmstetter, D. and Sornette, Diffusion of epicenters of earthquake aftershocks, Omori's law, and generalized continuous-time random walk models, The American Physical Society (2002).
\bibitem{Huelsenbeck} J.P. Huelsenbeck, B. Larget, D. Swofford,  A Compound Poisson Process for Relaxing the Molecular Clock. Genetics Society of America (2000).
\bibitem{KP00} G. Kerkyacharian, D. Picard, Thresholding algorithms, maxisets and well-concentrated bases, Test 9 (2000) 283--344.
\bibitem{CTRW} J. Masoliver, M. Montero, J. Perell\'o, G.H. Weiss, Direct and inverse problems with some generalizations and extensions,
Arxiv preprint 0308017v2 (2008).
\bibitem{Meerschaert06} M.M. Meerschaert, E. Scalas,  Coupled continuous time random walk in finance, Physica A 370 (2006) 114--118.
\bibitem{MOHARIR} P.S. Moharir, Estimation of the compounding distribution in the compound Poisson
process model for earthquakes, Proc. Indian Acad. Sci. 101 (1992) 347--359.
\bibitem{Reiss} M. Neumann, M. Rei{\ss}, Nonparametric estimation for L\'evy processes from low-frequency observations, Bernoulli 15 (2009) 223--248.
\bibitem{Scalas} E. Scalas, The application of continuous-time random walks in finance and economics, Physica A 362 (2006) 225--239.
\bibitem{VAN-ES07} B. van Es, S. Gugushvili, P. Spreij, A kernel type nonparametric density estimator for decompounding, Bernoulli 13 (2007) 672--694.
}
\end{thebibliography}
\end{document}